\documentclass{amsart}

\pdfoutput=1

\usepackage[dvipsnames,svgnames,x11names,hyperref]{xcolor}
\usepackage{bbm,url,graphicx,verbatim,amssymb,booktabs,lmodern,mathtools,microtype,mathabx}
\usepackage[shortlabels]{enumitem}
\usepackage[pagebackref,colorlinks,citecolor=Mahogany,linkcolor=Mahogany,urlcolor=Mahogany,filecolor=Mahogany]{hyperref}
\usepackage[capitalize]{cleveref}
\usepackage[mathscr]{euscript}
\usepackage[margin=1.38in]{geometry}
\usepackage{tikz,tikz-cd}
\usetikzlibrary{matrix,calc,positioning,arrows,decorations.pathreplacing,patterns,arrows}

\AtBeginDocument{%
	\def\MR#1{}
}

\setenumerate[1]{label=(\roman*),} 

\setitemize[1]{label=\raisebox{0.25ex}{\tiny$\bullet$}}

\newtheorem{theorem}{Theorem}[section]
\newtheorem{lemma}[theorem]{Lemma}
\newtheorem{proposition}[theorem]{Proposition}

\newtheorem*{corollary*}{Corollary}

\newtheorem{atheorem}{Theorem}

\theoremstyle{definition}
\newtheorem{definition}[theorem]{Definition}

\theoremstyle{remark}

\newtheorem{remark}[theorem]{Remark}
\newtheorem*{remark*}{Remark}

\newcommand{\icat}[1]{{\mathscr{#1}}}
\renewcommand{\rm}[1]{{\mathrm{#1}}}
\newcommand{\ra}{\rightarrow}
\newcommand{\lra}{\longrightarrow}

\DeclareMathOperator*{\colim}{colim}

\newcommand{\Fun}{\rm{Fun}}

\newcommand{\Map}{\rm{Map}}

\newcommand{\Cat}{{\icat{C}\rm{at}_\infty}}

\newcommand{\CAlg}{\rm{CAlg}}

\newcommand{\Env}{\rm{Env}}
\newcommand{\Finstar}{\rm{Fin}_*}
\newcommand{\Fin}{\rm{Fin}}
\newcommand{\Alg}{\rm{Alg}}
\newcommand{\Comm}{\rm{Comm}}

\newcommand{\Op}{{\icat{O}\rm{p}_\infty}}
\newcommand{\Opast}{{\icat{O}\rm{p}^\ast_\infty}}
\newcommand{\genOp}{{\icat{O}\rm{p}^{\rm{gen}}_\infty}}
\newcommand{\nuOp}{{\icat{O}\rm{p}^{\rm{nu}}_\infty}}
\newcommand{\Surj}{{\rm{Surj}_\ast}}
\newcommand{\nonuni}{\rm{nu}}

\newcommand{\ed}{E_d}

\newcommand{\Mul}{\rm{Mul}}
\newcommand{\id}{\rm{id}}
\newcommand{\Cocart}{\icat{C}\rm{ocart}}

\newcommand{\Fam}{\icat{F}\rm{am}}

\newcommand{\Ed}{E_d}

\newcommand{\bfR}{\mathbf{R}}
\newcommand{\bfF}{\mathbf{F}}
\newcommand{\bfQ}{\mathbf{Q}}
\newcommand{\Top}{\rm{Top}}
\newcommand{\BTop}{\rm{BTop}}

\DeclareFontFamily{U}{min}{}
\DeclareFontShape{U}{min}{m}{n}{<-> udmj30}{}

\title[Two remarks on spaces of maps between operads of little cubes]{Two remarks on spaces of maps between\\ operads of little cubes}

\author{Geoffroy Horel}
\address{Université Sorbonne Paris Nord, Laboratoire Analyse, Géométrie et Applications, CNRS (UMR 7539), 93430, Villetaneuse, France.}
\email{horel@math.univ-paris13.fr}

\author{Manuel Krannich}
\address{Department of Mathematics, Karlsruhe Institute of Technology, 76131 Karlsruhe, Germany}
\email{krannich@kit.edu}

\author{Alexander Kupers}
\address{Department of Computer and Mathematical Sciences, University of Toronto Scarborough, 1265 Military Trail, Toronto, ON M1C 1A4, Canada}
\email{a.kupers@utoronto.ca}

\begin{document}

\begin{abstract}We record two facts on spaces of derived maps between the operads $E_d$ of little $d$-cubes. Firstly, these mapping spaces are equivalent to the mapping spaces between the non-unitary versions of $E_d$. Secondly, all endomorphisms of $E_d$ are automorphisms. We also discuss variants for localisations of $E_d$ and for versions with tangential structures.
\end{abstract}
	
\maketitle

The operad of little $d$-cubes $\ed$, whose space $E_d(k)$ of $k$-ary operations is the space of rectilinear embeddings $\bigsqcup_k (-1,1)^d \hookrightarrow (-1,1)^d$, is omnipresent in homotopy theory. In recent years, it also gained prominence in geometric topology, not least because it became clear that the (derived) mapping space $\Map_{\Op}(\ed,E_{d'})$ from the $\ed$- to the $E_{d'}$-operad is closely related to spaces of embeddings of $d$- into $d'$-dimensional manifolds (see e.g.\,\cite{DwyerHess,Turchin,AroneTurchin,BoavidaWeiss}). This note serves to record two facts on these spaces of derived maps between $E_d$-operads.
 
\begin{remark*}We phrase the results in the $\infty$-category $\Op$ of $\infty$-operads in the sense of Lurie \cite{LurieHA}. However as $\Op$ is known to be equivalent to the underlying $\infty$-category of the model categories of other models of operads such as simplicial coloured operads \cite{cisinskimoerdijkdendroidal,cisinskimoerdijkdendroidal2,barwickoperator,chuhaugsengheuts}, we could have also stated the results in any of these settings.\end{remark*}

The first fact concerns the space of $0$-ary operations. The $E_d$-operad is \emph{unital} or \emph{unitary}, in the sense that it has a contractible space of $0$-ary operations; there is a unique embedding $\varnothing \hookrightarrow (-1,1)^d$. There is also \emph{non-unitary} variant $\Ed^\nonuni$, obtained by replacing the space of $0$-ary operations with the empty set. This is the value of $E_d$ under a ``non-unitarisation'' functor $(-)^{\rm{nu}}\colon \Op\ra \Op$, so there is in particular a comparison map from the space of maps $E_d\ra E_{d'}$ to the space of maps between the non-unitary variants. This is an equivalence:
\begin{atheorem}\label{athm:main} For $d,d' \geq 1$, the map \[(-)^{\rm{nu}}\colon \Map_{\Op}(\ed,E_{d'})\lra \Map_{\Op}(\ed^{\rm{nu}},E^{\rm{nu}}_{d'})\] is an equivalence.
\end{atheorem}
\begin{remark*}There are various extensions of \cref{athm:main}:
\begin{enumerate}
\item The target $E_{d'}$ may be replaced by any $\infty$-operad $\icat{O}$ such that for all colours $c$ the space of multi-operations $\Mul_{\icat{O}}(\varnothing,c)$ is non-empty and $\Mul_{\icat{O}}(c,c)$ is connected (see \cref{sec:main-result}). 
\item The source $\ed$ may be replaced by variants involving tangential structures, for instance the framed $\ed$-operad (see \cref{sec:tangential-structures}).
\item Source and target may be replaced by certain localisations (see \cref{sec:localisations}).
\item The mapping spaces can also be taken in the $\infty$-category underlying the model category of simplicial one-coloured operads instead of multi-coloured ones (see \cref{sec:one-coloured}).
\end{enumerate}
\end{remark*}
The second fact is that all endomorphisms of $\ed$ are automorphisms:
\begin{atheorem}\label{athm:end-is-aut} For $d \geq 1$, every self-map of $\ed$ is an equivalence.
\end{atheorem}

\subsection*{Acknowledgements} We thank Rune Haugseng, Gijs Heuts, and Jan Steinebrunner for helpful discussions. This material is partially based upon work supported by the Swedish Research Council under grant no.\,2016-06596 while the authors were in residence at Institut Mittag-Leffler in Sweden during the semester \emph{Higher algebraic structures in algebra, topology and geometry}.

GH acknowledges the support of the French Agence Nationale pour la Recherche (ANR) project number ANR-20-CE40-0016 HighAGT as well as the Romanian Ministry of Education and Research, grant CNCS-UE-FISCDI, project number PN-III-P4-ID-PCE-2020-279.

AK acknowledges the support of the Natural Sciences and Engineering Research Council of Canada (NSERC) [funding reference number 512156 and 512250], as well as the Research Competitiveness Fund of the University of Toronto at Scarborough. AK is supported by an Alfred J.~Sloan Research Fellowship.

\section{\cref{athm:main} and a generalisation}

\subsection{Non-unitary operads}
To define \emph{non-unitary $\infty$-operads}, we denote the category of finite pointed sets by $\Fin_\ast$ and write $\iota\colon \Surj\hookrightarrow \Finstar$ for the inclusion of the wide subcategory on the surjections. This inclusion can be obtained by taking operadic nerves \cite[2.1.1.27]{LurieHA} of the inclusion $\Comm^{\rm{nu}}\hookrightarrow \Comm$ of one-coloured simplicial operads whose spaces of $k$-ary operations consist of a point in both cases for $k\ge1$, and are for $k=0$ given by $\Comm(0)=\ast$ and $\Comm^{\rm{nu}}(0)=\varnothing$. In particular, $\iota\colon \Surj\hookrightarrow \Finstar$ is an $\infty$-operad. The following definition appears implicitly in \cite[5.4.4.1]{LurieHA}.

\begin{definition}\label{dfn:non-unitary}An $\infty$-operad $\icat{O}$ is \emph{non-unitary} if the map $\icat{O}^{\otimes}\ra \Fin_\ast$ factors over $ \Surj\hookrightarrow \Finstar$. We denote the full subcategory of \emph{non-unitary $\infty$-operads} by $\nuOp\subset\Op$.
\end{definition}

\begin{remark}\ 
\begin{enumerate}
\item The forgetful functor $\smash{\Op_{/\Surj}}\ra \Op$ of the category of $\infty$-operads over $ \Surj\hookrightarrow \Finstar$ lands in the subcategory $\nuOp\subset\Op$, and since factorisations of maps $\icat{O}^{\otimes}\ra \Fin_\ast$ over $\Surj\hookrightarrow \Finstar$ are unique, the resulting functor $\Op_{/\Surj}\ra\nuOp$ is an equivalence.
\item Equivalently, an $\infty$-operad $\icat{O}$ is non-unitary if its spaces of multi-operations \cite[2.1.1.16]{LurieHA} satisfy $\Mul_{\icat{O}}(\varnothing,c)=\varnothing$ for all colours $c$ (this follows straight from the axioms of an $\infty$-operad \cite[2.1.1.10]{LurieHA}).
\item Guided by \cite[5.4.4.1]{LurieHA}, one might be tempted to use the adjective ``non-unital'' as opposed to ``non-unitary''. We opted against it since, firstly, ``non-unital'' is used in \cite[2.3]{LurieHA} for a weaker condition and, secondly, \cref{dfn:non-unitary} is consistent with \cite{FresseBook} in that a non-unitary $\infty$-operad is the multi-coloured and $\infty$-categorical version of the notion of a non-unitary operad from loc.cit.
\end{enumerate}\end{remark}

As mentioned in the previous remark, the inclusion $\iota_\ast\colon \nuOp \hookrightarrow \Op$ can be viewed as the forgetful functor $\smash{\Op_{/\Surj}}\ra \Op$, so it has (as any forgetful functor of an overcategory of a category with products) a right-adjoint $\iota^\ast\colon \Op\ra \nuOp$ given by taking products with $\Surj$ in $\Op$. As the forgetful functor $\Op \to \Cat_{/\Finstar}$ preserves products (it in fact creates all limits \cite[1.13]{AyalaFrancisTanaka}), this right-adjoint is given by sending an operad $\icat{O}^\otimes\ra \Finstar$ to the pullback $\icat{O}^\otimes\times_{\Finstar}\Surj\rightarrow \Surj$. We write 
\[(-)^{\rm{nu}}\colon \Op\lra \Op\]
for the composition $(-)^{\rm{nu}}\coloneq \iota_*\iota^*$. This is a colocalisation since $\iota_\ast\colon \nuOp \hookrightarrow \Op$ is fully faithful. 

\subsection{Statement and proof of a generalisation of \cref{athm:main}}\label{sec:main-result}
We consider the following property on the spaces of multi-operations of an $\infty$-operad $\icat{O}$:

\begin{definition}\label{dfn:quasi-unitalising}An $\infty$-operad $\icat{O}$ is \emph{quasi-unitalising} if for each colour $c$, the space $\Mul_{\icat{O}}(\varnothing,c)$ is non-empty and $\Mul_{\icat{O}}(c,c)$ is connected.
\end{definition}

The operads $\ed$ are one-coloured and have contractible spaces of $0$- and $1$-ary operations, so in particular are quasi-unitalising. \cref{athm:main} is thus special case of the following result. Its statement involves the notion of a $0$-coconnected map, which is a map that induces an injection on the level of path-components and an isomorphism on all homotopy groups of degree $i\ge1$.

\begin{theorem}\label{thm:main-general-target}
	For $d\ge1$ and any $\infty$-operad $\icat{O}$, the map 
	\[(-)^{\nonuni}\colon \Map_{\Op}(\ed,\icat{O}) \lra \Map_{\Op}(E^\rm{nu}_d,\icat{O}^\rm{nu})\]
	is $0$-coconnected. If $\icat{O}$ is quasi-unitalising, then it is an equivalence.
\end{theorem}

\begin{proof}
Using that $(-)^\rm{nu}$ is a colocalisation, it suffices to show the claim for the map \begin{equation}\label{equ:precomposition-map}j^*\colon \Map_{\Op}(\ed,\icat{O}) \lra \Map_{\Op}(E^\rm{nu}_d,\icat{O})
\end{equation} 
obtained by precomposition with the counit $j\colon \ed^\rm{nu}\rightarrow \ed$.

To start with, we consider the more restrictive case where $\icat{O}=\icat{C}$ is a symmetric monoidal $\infty$-category as opposed to a general $\infty$-operad. In this case the claim can be extracted from \cite{LurieHA}: it follows directly from the definition of the $\infty$-category $\Op$  \cite[2.1.4.1]{LurieHA} that the map in question agrees with the map
\begin{equation}\label{equ:map-of-alg-cats}
	j^* \colon \Alg_{\ed}(\icat{C})^\simeq \lra \Alg_{\ed^\rm{nu}}(\icat{C})^{\simeq}
\end{equation}
obtained by applying cores to the functor of $\infty$-categories $\Alg_{\ed}(\icat{C})\rightarrow \smash{\Alg_{\ed^\rm{nu}}}(\icat{C})$ induced by precomposition with $j$; here $\Alg_{\icat{P}}(\icat{C})$ for an $\infty$-operad $\icat{P}$ denotes as in \cite[2.1.2.7]{LurieHA} the $\infty$-category of $\icat{P}$-algebras in $\icat{C}$. By an application of \cite[5.4.4.5]{LurieHA} to the cocartesian fibration $\icat{C}^{\otimes}\times_{\Finstar}\ed^{\otimes}\rightarrow \ed^{\otimes}$ (see also the top of page 946 loc.cit.), this functor is equivalent to the inclusion $\smash{\Alg_{\ed^\rm{nu}}^{\rm{qu}}(\icat{C})\hookrightarrow \Alg_{\ed^\rm{nu}}(\icat{C})}$ of the subcategory of \emph{quasi-unital} algebras: those $\smash{\ed^\rm{nu}}$-algebras $A$ in $\icat{C}$ whose underlying non-unital associative algebra admits a \emph{quasi-unit}, i.e.\,there is a map $u\colon 1_{\icat{C}}\rightarrow A$ from the monoidal unit such that the compositions
\begin{equation}\label{equ:unit-maps}
	A\simeq 1_{\icat{C}}\otimes A\xrightarrow{u\otimes \id_A}A\otimes A\xrightarrow{\mu} A\quad\text{and}\quad A\simeq 1_{\icat{C}}\otimes A\xrightarrow{\id_A\otimes u}A\otimes A\xrightarrow{\mu} A, 
\end{equation} 
involving the multiplication $\mu$ of $A$, are homotopic to the identity. Morphisms in $\smash{\Alg_{\ed^\rm{nu}}^{\rm{qu}}(\icat{C})}$ are those morphism $f\colon A\ra B$ of $\smash{\ed^\rm{nu}}$-algebras such that for a choice of quasi-unit $u$ of $A$ the composition $(f\circ u)$ is a quasi-unit for $B$. Note that the latter condition is always satisfied if $f$ is an equivalence, so the functor \eqref{equ:map-of-alg-cats} on cores is fully faithful. This implies the first part of the claim since a fully faithful functor between $\infty$-groupoids corresponds via the equivalence between $\infty$-groupoids and spaces to an ``inclusion of path-components'', i.e.\,a $0$-coconnected map. The second part of the claim is equivalent to showing that \eqref{equ:map-of-alg-cats} is essentially surjective if $\icat{O}$ is quasi-unitalising which, by the description of this functor in terms of quasi-unital algebras just explained, is equivalent to showing that any $\smash{\ed^\rm{nu}}$-algebras $A$ in $\icat{C}$ admits a quasi-unit. The first condition in being quasi-unitalising implies that there is \emph{some} map $u\colon1_{\icat{C}}\rightarrow A$, and any such map is a quasi-unit since the second condition implies that \emph{any} self-map of $A$ is homotopic to the identity, so in particular the two in \eqref{equ:unit-maps}.

To extend this argument to the case of a general $\infty$-operad $\icat{O}$, we use that the inclusion $\CAlg(\Cat)\hookrightarrow \Op$ of symmetric monoidal $\infty$-categories into $\infty$-operads has a left-adjoint, the \emph{monoidal envelope} $\Env(-)\colon\Op\rightarrow \CAlg(\Cat)$ from \cite[2.2.4]{LurieHA}. Since $\Finstar$ is the terminal $\infty$-operad, this functor lifts to a functor on overcategories
\begin{equation}\label{equ:better-envelope} \Op\simeq \Op_{/\Finstar}\longrightarrow \CAlg(\Cat)_{/\Env(\Finstar)}\simeq \CAlg(\Cat)_{/\Fin}
\end{equation}
which we denote by the same symbol. Here we used that the envelope $\Env(\Finstar)$ of the terminal operad is equivalent to the category $\Fin$ of finite sets with the cocartesian monoidal structure \cite[2.3.7]{HaugsengKock}. Moreover, by \cite[2.4.3]{HaugsengKock}, the lifted functor \eqref{equ:better-envelope} is fully faithful, so the map in the statement is equivalent to the map
\[\Env(j)^*\colon \Map_{\CAlg(\Cat)_{/\Fin}}(\Env(\ed),\Env(\icat{O})) \lra \Map_{\CAlg(\Cat)_{/\Fin}}(\Env(E^\rm{nu}_d),\Env(\icat{O}))\]
which is---by adjunction---in turn equivalent to the map
\[j^*\colon \Map_{\Op_{/\Fin}}(\ed,\Env(\icat{O})) \lra \Map_{\Op_{/\Fin}}(\ed^\rm{nu},\Env(\icat{O})).\]
Since mapping spaces in overcategories are computed as fibres of the corresponding mapping spaces in the non-overcategories, this map is the map on vertical fibres of the square
\begin{equation}\label{equ:square-overcategory}
	\begin{tikzcd}[row sep=0.4cm]
	\Map_{\Op}(\ed,\Env(\icat{O}))\rar{j^*}\dar{}&\Map_{\Op}(\ed^\rm{nu},\Env(\icat{O}))\dar\\
	\Map_{\Op}(\ed,\Fin)\rar{j^*}&\Map_{\Op}(\ed^\rm{nu},\Fin).
	\end{tikzcd}
\end{equation}
Here the vertical arrows are induced by postcomposition with the map obtained by applying $\Env(-)$ to the unique map of operads $\icat{O}\rightarrow \Fin_*$, and the vertical fibres are taken at the analogous maps for $\ed$ and $\ed^\rm{nu}$. Since $\Env(\icat{O})$ and $\Fin$ are symmetric monoidal, both horizontal maps are $0$-coconnected as an instance of \eqref{equ:map-of-alg-cats}, so the map on fibres is $0$-coconnected as well. This finishes the proof of the first part of the claim.

This leaves us with showing the second part of the claim in the general case. For this we note that since $\Fin$ is cocartesian and the underlying $\infty$-category of colours of $\ed$ is trivial since the space of $1$-ary operations is contractible, an application of \cite[2.4.3.9]{LurieHA} shows that $\Map_{\Op}(\ed,\Fin)$ is equivalent to the core $\Fin^{\simeq}$, and with respect to this equivalence the vertical fibres are taken at $\{1\}\in\Fin^\simeq$. Since the component of $\{1\}\in \Fin^\simeq$ is contractible and the bottom horizontal map in \eqref{equ:square-overcategory} is $0$-coconnected, it suffices to show that the upper horizontal arrow in \eqref{equ:square-overcategory} hits those components in $\Map_{\Op}(\ed^\rm{nu},\Env(\icat{O}))$ that map to the image of $\{1\}$ under the bottom horizontal arrow. By the description of the components hit by \eqref{equ:map-of-alg-cats} given earlier, this is equivalent to showing that any non-unital associative algebra $A$ in $\Env(\icat{O})$ that maps via the vertical map to the image of $\{1\}\in\Fin$ with respect to the bottom horizontal map admits a quasi-unit. In order to see this, let us recall the description of the homotopy category of $\Env(\icat{O})$ from \cite[2.2.4.3]{LurieHA}: objects are given by a pair $(S,(c_s)_{s\in S\backslash\{*\}})$ of a finite pointed set $S\in\Fin_*$ and a sequence $c_s$ of objects in the underlying category of colours of $\icat{O}$. Morphisms $(S,(c_s)_{s\in S\backslash\{*\}})\rightarrow (T,(d_{t})_{t\in T\backslash\{*\}})$ are given by an active map $f\colon S\rightarrow T$ (i.e.\,$f^{-1}(*)=*$) and multioperations $g_t\in \Mul_{\icat{O}}((c_s)_{s\in f^{-1}(t)},d_t)$ for $t\in T\backslash\{*\}$. The map to $\Fin$ sends $(S,(c_s)_{s\in S\backslash\{*\}})$ to $S\backslash\{*\}$. The composition and the monoidal structure are given in the evident way. Now if we are given a non-unital associative algebra in $\Env(\icat{O})$ that maps to $\{1\}\in\Fin$, then the underlying object has the form $(\{1,*\},c)$ and the multiplication is given by a multi-operation $\mu\in\Mul_{\icat{O}}((c,c),c)$. To provide a quasi-unit, it thus suffices to give an element $u\in\Mul_{\icat{O}}(\varnothing,c)$ such that the two maps analogous to \eqref{equ:unit-maps} (using operadic composition) are homotopic to the identity in $\Mul_{\icat{O}}(c,c)$. The operad $\icat{O}$ being quasi-unitalising means that $\Mul_{\icat{O}}(\varnothing,c)$ is non-empty and $\Mul_{\icat{O}}(c,c)$ is connected, so this is always possible and the claim follows.\end{proof}

\section{Further extensions of \cref{athm:main}}
This section serves to explain several extensions of \cref{thm:main-general-target}. Firstly, in \cref{sec:tangential-structures}, we extend the result to allow variants of $\ed$ that include tangential structures (including the framed $\ed$-operad).  Secondly, in \cref{sec:localisations}, we extend the result to allow localisations of $\ed$. Thirdly, in \cref{sec:one-coloured} we extend the result to mapping spaces of one-coloured operads.

\subsection{Versions of $\ed$ with tangential structures}\label{sec:tangential-structures}
Recall that the $\infty$-operad $\ed$ is obtained as the operadic nerve of the one-coloured simplicial operad whose space $\ed(k)$ of $k$-ary operations is the space of rectilinear embeddings $\bigsqcup_k (-1,1)^{d}\hookrightarrow (-1,1)$. Instead of rectilinear embeddings, one may use all topological embeddings to define a related $\infty$-operad $\smash{\ed^{\Top}}$, which is denoted $\BTop(d)^{\otimes}\ra \Finstar$ in \cite[5.4.2.1]{LurieHA} since its underlying $\infty$-category of colours is equivalent to the classifying space $\BTop(d)$ of the topological group of homeomorphisms of $\bfR^d$ as a result of the Kister--Mazur theorem \cite[5.4.2.6]{LurieHA}. To define yet another $\infty$-operad, one may use $k$-tuples of self-embeddings of $(-1,1)^{d}$ instead of topological embeddings $\bigsqcup_k(-1,1)^{d}\hookrightarrow (-1,1)^{d}$. The resulting $\infty$-operad is equivalent to the cocartesian $\infty$-operad $\BTop(d)^{\sqcup}$ associated to $\BTop(d)$ \cite[2.4.3]{LurieHA}. An embedding $\bigsqcup_k (-1,1)^{d}\hookrightarrow (-1,1)^d$ is in particular a $k$-tuple of self-embeddings $(-1,1)^{d}$, so there is a map of $\infty$-operads $\smash{\ed^{\Top}}\ra \BTop(d)^{\sqcup}$. A map $\theta\colon B\ra \BTop(d)$ of spaces induces a map $B^{\sqcup}\ra \BTop(d)^{\sqcup}$ of cocartesian $\infty$-operads, so we may take the pullback 
\[\ed^{\theta}\coloneq \ed^{\Top}\times_{\BTop(d)^{\sqcup}}B^{\sqcup}\]
in $\infty$-operads. We call this $\infty$-operad the \emph{$\theta$-framed $\ed$-operad}. If $B$ is connected, it can equivalently be constructed as the operadic nerve of a one-coloured simplicial operad involving $\theta$-framed topological embeddings $\bigsqcup_k (-1,1)^{d}\hookrightarrow (-1,1)$. For $\theta=(\ast\ra \BTop(d))$, this recovers $\ed$ and for $\theta=(\rm{BSO}(d)\ra \BTop(d))$ this operad is unfortunately known as the \emph{framed little $d$-discs operad}. In this subsection, we generalise \cref{thm:main-general-target} to the $\theta$-framed case for any $\theta$:

\begin{theorem}\label{thm:framed-version}For $d\ge1$, a map $\theta\colon B \to B\Top(d)$ of spaces, and an $\infty$-operad $\icat{O}$, the map 
\[\Map_{\Op}(\ed^\theta,\icat{O}) \lra \Map_{\Op}(\ed^{\theta,\nonuni},\icat{O}^\nonuni)\]
is $0$-coconnected. If $\icat{O}$ is quasi-unitalising then this map is an equivalence.
\end{theorem}

We will deduce \cref{thm:framed-version} from \cref{thm:main-general-target} by means of the following proposition:

\begin{proposition}\label{prop:nu-colim}Let $d\ge1$ and $\theta\colon B \to B\Top(d)$ a map of spaces. 
\begin{enumerate}
	\item\label{enum:assembly-i} There is a functor $G_\theta\colon B\ra \Op$ whose values are equivalent to $\ed$ and which satisfies $\colim_{b\in B}G_\theta(b)\simeq \ed^\theta$.
	\item\label{enum:assembly-ii} The canonical map $\colim_{b\in B}(G_\theta(b)^{\rm{nu}})\ra (\colim_{b\in B}G_\theta(b))^{\rm{nu}}$ is an equivalence.
\end{enumerate}
\end{proposition}

\begin{proof}[Proof of \cref{thm:framed-version} assuming \cref{prop:nu-colim}] Firstly, by the colocalisation property of $(-)^{\rm{nu}}$ it suffices to show the claim for the map $\Map_{\Op}(\ed^\theta,\icat{O}) \ra \Map_{\Op}(\ed^{\theta,\nonuni},\icat{O})$ induced by precomposition with the counit $\smash{\ed^{\theta,\nonuni}\ra \ed^\theta}$. By \cref{prop:nu-colim}, this counit is a colimit of maps that are equivalent to the counit $\ed^{\nonuni}\ra \ed$ for which we already known the claim by \cref{thm:main-general-target}, so \cref{thm:framed-version} follows from the universal property of the colimit.\end{proof}

\begin{proof}[Proof of \cref{prop:nu-colim}]
We begin by recalling the point of view on colimits of $\Op$-valued functors via families of operads. For simplicity (and because it is all we need) we restrict to the case of functors $G\colon X\ra \Op$ defined on an $\infty$-groupoid $X$ as opposed to a general $\infty$-category. Consider the following commutative diagram of $\infty$-categories
\begin{equation}\label{equ:colim-is-assem}
	\begin{tikzcd}[ar symbol/.style = {draw=none,"\textstyle#1" description,sloped},	subset/.style = {ar symbol={\subset}}, row sep=0.4cm]
	\Fun(X,\Cat_{/\Finstar})\rar{\simeq}[swap]{\rm{unstr}}&\Cat_{/X\times \Finstar}&[-1cm]\\
	\Fun(X,\Op)\arrow[u,hookrightarrow]\arrow[dr,"\text{colim}",swap, bend right=5]\rar{\simeq}&\Fam(X)\arrow[u,hookrightarrow]\arrow[r,subset]&[-1cm](\genOp)_{/X\times\Finstar}\rar{\text{forget}}&\genOp\arrow[dll,"\rm{assem}",bend left=5]\\[-0.2cm]
	&\Op&&
	\end{tikzcd}
\end{equation}
The upper row is given by the unstraightening equivalence, which restricts to an equivalence between the subcategory $\Fun(X,\Op)$ of $\Fun(X,\Cat_{/\Finstar})$ and the subcategory $\Fam(X)$ of $\Cat_{/X\times \Finstar}\simeq \Cocart(X)_{/{(X\times \Finstar\ra X)}}$ whose objects are those functors $\icat{C}\ra X\times \Finstar$ that are families of operads indexed by $X$ in the sense of \cite[2.3.2.10]{LurieHA} and whose morphisms are those maps over $X\times \Finstar$ that preserve cocartesian lifts of inert morphisms in $\Fin$ (see the discussion in \cite[Section 2.11]{HinichYoneda}; note that any family of operads indexed by $X$ is cocartesian in Hinich's sense since $X$ is an $\infty$-groupoid).  The $\infty$-category $\Fam(X)$ can be identified with a full subcategory of the overcategory $(\genOp)_{/X\times\Finstar}$ of the $\infty$-category $\genOp$ of generalised operads in the sense of \cite[2.3.2.1-2.3.2.2]{LurieHA}, over the projection $\rm{pr}\colon X\times\Finstar\ra \Finstar$ in $\genOp$ \cite[2.3.2.13]{LurieHA}. The functor labelled \text{assem} is Lurie's assembly construction which is the left-adjoint to the full subcategory inclusion $\Op\subset\genOp$ \cite[2.3.3.3]{LurieHA}. This explains the diagram, except for the commutative of the lower triangle which---by the universal property of the colimit---follows from the sequence of equivalences
\begin{align*} 
	\Map_{\Fun(X,\Op)}(G,\rm{const}_\icat{O}) &\simeq \Map_{\Fam(X)}(\rm{unstr}(G),X \times \icat{O}) && \\
	&\simeq \Map_{(\genOp)_{/X \times \Finstar}}(\rm{unstr}(G),X \times \icat{O}) && \text{\cite[2.3.2.13]{LurieHA}}\\
	&\simeq \Map_{\genOp}(\rm{unstr}(G),\icat{O}) && \\
	&\simeq \Map_\Op(\rm{assem}(\rm{unstr}(G)),\icat{O}) && \text{\cite[2.3.3.3]{LurieHA}.}
\end{align*}
which is natural in $G\in \Fun(X,\Op)$ and $\icat{O}\in\Op$. Note that by the naturality of unstraightening, the value of $G\colon X\ra \Op$ at $x\in X$ corresponds to the pullback of the corresponding family $\rm{unstr}(G)\ra X\times\Finstar$ along $\{x\}\times\Finstar\hookrightarrow  X\times\Finstar$.

Equipped with \eqref{equ:colim-is-assem} we now turn to the proof of the first part of the claim. Since the underlying $\infty$-category of colours of $\smash{\ed^{\Top}}$ is $\BTop(d)$, so an $\infty$-groupoid, the proof of \cite[2.3.4.4]{LurieHA} produces a map of generalised $\infty$-operads $\smash{\widetilde{E}_d^{\Top}\ra E_d^{\Top}}$ where $\smash{\widetilde{E}_d^{\Top}}$ is the total space of a family of $\infty$-operads indexed by $\BTop(d)$: \[\smash{\widetilde{E}_d^{\Top}}\lra\BTop(d)\times\Finstar.\] The cited proof also shows that this map of generalised $\infty$-operads is an approximation in the sense of \cite[2.3.3.6]{LurieHA}, and \cite[5.4.2.9]{LurieHA} shows that the fibres of the family $\smash{\widetilde{E}_d^{\Top}}$ indexed by $\BTop(d)$ are equivalent to $E_d$. By pulling back along $B\ra \BTop(d)$ and using that approximations are preserved by pullbacks \cite[2.3.3.9]{LurieHA}, we obtain an analogous approximation $\smash{\widetilde{E}_d^{\theta}\ra E_d^{\theta}}$ to $\smash{E_d^{\theta}}$ by a family of operads $\smash{\widetilde{E}_d^{\theta}}$ indexed by $B$ whose fibres are equivalent to $E_d$. Under the equivalence $\Fun(B,\Op)\simeq\Fam(B)$ from \eqref{equ:colim-is-assem}, the family $\smash{\widetilde{E}_d^{\theta}}$ corresponds to a functor $G_\theta\in \Fun(B,\Op)$ whose values are equivalent to $E_d$. Moreover, commutativity of \eqref{equ:colim-is-assem} implies the first equivalence in the sequence 
\begin{equation}\label{equ:colim-presentation}
	\colim G_\theta\simeq \rm{assem}(\widetilde{E}_d^{\theta})\simeq \smash{E_d^{\theta}};
\end{equation}
the second equivalence follows from \cite[2.3.4.5 (1), Proof of 2.3.4.4]{LurieHA}. This proves \ref{enum:assembly-i}.

To prove \ref{enum:assembly-ii}, we first note that there is a variant of the upper-left square in \eqref{equ:colim-is-assem} where one replaces the category $\Finstar$ by $\Surj$, the category $\Op$ by $\nuOp$ and $\Fam(X)$ by the $\infty$-category $\Fam_{\Surj}(X)$ of $\Surj$-families of operads indexed by $X$ in the sense of \cite[2.11]{HinichYoneda} if one makes $X\times \Surj$ into a decomposition category as in \cite[2.11.1]{HinichYoneda}. Now consider the commutative diagram of $\infty$-categories
\begin{equation}\label{equ:nu-on-families}
	\begin{tikzcd}[column sep=2cm,row sep=0.5cm]
	\Fun(B,\Op)\arrow[r,"{\iota^*}"]\dar{\simeq}\arrow[rr,"{(-)^{\rm{nu}}}", bend left=10]& \Fun(B,\nuOp)\arrow[r,"{\iota_*}"]\dar{\simeq}&\Fun(B,\Op)\dar{\simeq}\\
	\Fam(B)\rar{\iota^*}&\Fam_{\Surj}(B)\rar{\iota_*}&\Fam(B)
	\end{tikzcd}
\end{equation}
where the left horizontal arrows are induced by pullback along $\iota\colon \Surj\hookrightarrow\Fin$ and $\id_B\times\iota$ respectively, and the right horizontal arrows by postcomposition with $\iota$ and $\id_B\times\iota$ respectively. The right horizontal arrows are the respective left-adjoints to the left horizontal arrows (see \cite[2.6.6]{HinichYoneda} for the lower row). Now consider the pullback square
\begin{equation}\label{equ:pullback-approx}
	\begin{tikzcd}[column sep=0.5cm,row sep=0.5cm]
	j^* \widetilde{E}^{\theta}_d\rar\dar& \ed^{\theta,\rm{nu}}\dar{j}\\
	\widetilde{E}_d^{\theta}\rar&\ed^\theta
	\end{tikzcd}
\end{equation}
where $j$ is the counit of the $(\iota_\ast,\iota^\ast)$-adjunction of endofunctors on $\Op$. Since this counit is by construction the  pullback inclusion $\iota^*E_d^{\theta}\ra E_d^{\theta}$ viewed as a map of operads, the left vertical arrow is the analogous pullback inclusion $\iota^*\widetilde{E}_d^{\theta}\ra \widetilde{E}_d^{\theta}$ viewed as a map of families of operads indexed by $B$. The latter agrees with the counit of the $(\iota_\ast,\iota^\ast)$-adjunction of endofunctors on $\Fam(B)$, so it agrees in view of \eqref{equ:nu-on-families}, via the equivalence $\Fam(B)\simeq \Fun(B,\Op)$, with the counit inclusion $((-)^{\rm{nu}}\circ G_\theta)\ra G_\theta$ of the $(\iota_\ast,\iota^\ast)$-adjunction of endofunctors on $\Fun(B,\Op)$. Using commutativity of \eqref{equ:colim-is-assem}, taking adjoints in \eqref{equ:pullback-approx} thus induces a commutative diagram
\[
	\begin{tikzcd}[ar symbol/.style = {draw=none,"\textstyle#1" description,sloped},	equ/.style = {ar symbol={\simeq}},column sep=0.5cm,row sep=0.5cm]
	\rm{colim}((-)^{\rm{nu}}\circ G_\theta)\dar\arrow[r,equ]&[-0.3cm]\rm{assem}(j^* \widetilde{E}^{\theta}_d)\rar\dar& E_d^{\theta,\rm{nu}}\dar{j}\\
	\rm{colim}(G_\theta)\arrow[r,equ]&\rm{assem}(\widetilde{E}_d^{\theta})\rar{\simeq}&\ed^\theta
	\end{tikzcd}
\]
whose bottom right equivalence featured in \eqref{equ:colim-presentation}. To show the claim it thus suffices to show that the upper right arrow is an equivalence. This follows from \cite[2.3.4.5 (1), Proof of 2.3.4.4]{LurieHA}, since the top arrow in \eqref{equ:pullback-approx} is an approximation because the bottom arrow is an approximation by construction and approximations are pullback-stable \cite[2.3.3.9]{LurieHA}.
\end{proof}

\subsection{Localised versions of $\ed$}\label{sec:localisations} The following is a direct consequence of \cref{thm:framed-version} and the universal property of localisations.

\begin{theorem}\label{thm:local-version}For $d\ge1$, a map of spaces $\theta\colon B\ra \rm{BTop}(d)$, an $\infty$-operad $\icat{O}$, and a localisation $L_* \colon \Op \to \Op$ commuting with $(-)^{\rm{nu}}$, the map 
	\[(-)^{\rm{nu}}\colon \Map_{\Op}(L_*E_d^{\theta},L_*\icat{O}) \lra \Map_{\Op}(L_*E_d^{\theta,\rm{nu}},L_*\icat{O}^\nonuni)\]
is $0$-coconnected. If $L_*\icat{O}$ is quasi-unitalising, then it is an equivalence.
\end{theorem}

\begin{remark}A source of localisations as in \cref{thm:local-version} is the following. By definition, a (reflective) localisation $L \colon \icat{S} \to \icat{S}$ of the $\infty$-category $\icat{S}$ of spaces is given by precomposing a fully faithful right-adjoint $R_0 \colon \icat{S}_0 \to \icat{S}$ with left-adjoint $L_0 \colon \icat{S} \to \icat{S}_0$. If $L_0$ preserves finite products then so does $R_0$, and then both $L_0$ and $R_0$ are symmetric monoidal with respect to the cartesian monoidal structures. As a consequence of \cite[Proposition 3.5.10]{ChuHaugseng}, they then induce on categories of enriched $\infty$-operads a fully faithful right adjoint $(R_0)_* \colon \Op(\icat{S}_0) \to \Op(\icat{S}) = \Op$ with left adjoint $(L_0)_* \colon \Op = \Op(\icat{S}) \to \Op(\icat{S}_0)$. In particular, the composition $L_* = (R_0)_* \circ (L_0)_* \colon \Op \to \Op$ is a localisation. On spaces of multi-operations, this is given by applying $L$, so $L_*$ commutes with $(-)^\nonuni$ if $L$ preserves the empty set and $L_*$ preserves the property of being quasi-unitalising if furthermore $L$ preserves connected spaces.\end{remark}

\begin{remark}
	For rationalisation, this gives a conceptual reason for the observation of Fresse--Willwacher \cite[Section 7]{FresseWillwacher} that their models for the automorphism spaces of the unitary and non-unitary versions of the rationalised $E_d$-operad $(\ed)_{\mathbf{Q}}$ agree. 
\end{remark}

\subsection{The one-coloured version of $\ed$}\label{sec:one-coloured} So far we worked in the $\infty$-category of $\infty$-operads $\Op$ which is, as mentioned in the introduction, equivalent to the underlying $\infty$-category of the model category of coloured simplicial operads. However, for some applications, the $\infty$-category $\Opast$ underlying the model category of \emph{one-coloured} simplicial operads plays a role. There is an evident forgetful functor $\Opast\to \Op$ which is---analogous to the situation of comparing simplicial groups with simplicial groupoids---\emph{not} fully faithful: this functor factors through the slice category $(\Op)_{*/}$ over the one-coloured operad  $*$ with only the identity operation since $*$ is initial in $\Opast$, and it is the resulting functor $\Opast\to \Op_{*/}$ that is fully faithful instead:

\begin{lemma}\label{lem:forgot-ff} 
	The forgetful functor $\Opast\to \Op_{*/}$ is fully faithful.
\end{lemma}

\begin{proof}
	Denoting by $\mathrm{Op}$ and $\mathrm{Op}^*$ the model categories of simplicial coloured operads and simplicial one-coloured operads respectively, the forgetful functor $\mathrm{Op}^* \ra (\mathrm{Op})_{*/}$ has a right adjoint which sends a simplicial coloured operad under $*$ to the full suboperad whose only colour is the one in the image of $*$. Clearly, both adjoints preserve weak equivalences, so it follows that the left adjoint induces a functor $\Opast \to (\Op)_{*/}$, which can be identified with the functor in the statement. The claim now follows from the fact that the counit of the adjunction is an isomorphism, so in particular a weak equivalence.
\end{proof}

Being the operadic nerve of a one-coloured simplicial operad when $B$ is connected, $\ed^\theta$ may be considered as an object in $\Opast$. The analogue of \cref{thm:framed-version} in this setting reads as follows:
\begin{theorem}\label{thm:one-colour}
For $d \geq 1$, a map of connected spaces $\theta\colon B\ra \rm{BTop}(d)$, and a one-coloured simplicial operad $\icat{O}$, the map
\[(-)^{\rm{nu}}\colon \Map_{\Opast}(\ed^\theta,\icat{O})\to \Map_{\Opast}(\ed^{\theta,\nonuni},\icat{O}^{\nonuni})\]
is $0$-coconnected. If $\icat{O}(0)$ is nonempty and $\icat{O}(1)$ is connected, then this map is an equivalence.
\end{theorem}

\begin{proof}	
Both rows in the commutative diagram
\[\begin{tikzcd}[column sep=0.5cm,row sep=0.5cm] \Map_{\Opast}(\ed^\theta,\icat{O}) \dar \rar & \Map_\Op(\ed^\theta,\icat{O}) \dar\rar &\Map_{\Op}(*,\icat{O})\dar\\
\Map_{\Opast}(\ed^{\theta,\nonuni},\icat{O}^\nonuni)  \rar & \Map(\ed^{\theta,\nonuni},\icat{O}^\nonuni)\rar&\Map_{\Op}(*,\icat{O}^\nonuni), \end{tikzcd}\]
are fibre sequences as a result of \cref{lem:forgot-ff} and the fact that mapping spaces in an under-$\infty$-category are the fibres of the respective mapping spaces in the non-under-$\infty$-categories. In view of this, the claim follows from the fact that the middle and right vertical maps are equivalences: the former by \cref{thm:framed-version}, and the latter since $\Map_{\Op}(*,\icat{O})$ and $\Map_{\Op}(*,\icat{O}^\nonuni)$ are both equivalent to the components $\icat{O}(1)^{\simeq}\subseteq \icat{O}(1)$ that are invertible under composition.
\end{proof}

\begin{remark}
The case $d=1$ of \cref{thm:one-colour} was proved by Muro \cite[p.\,2146]{murohomotopy}.
\end{remark}

\section{\cref{athm:end-is-aut}}
We conclude by proving \cref{athm:end-is-aut}: any endomorphism of $\ed$ is an equivalence.
\begin{proof}[Proof of \cref{athm:end-is-aut}]
It suffices to show that any self-map $\varphi\colon \ed\rightarrow\ed$ induces an equivalence on the space $\ed(k)$ of $k$-ary operations for all $k\ge0$. Recall that $\ed(k)$ is equivalent to the space of $k$ ordered configurations in $\bfR^d$. The claim for $d=1$ follows from the fact that $\Sigma_k$-equivariant self-maps of $E_1(k)\simeq\Sigma_k$ are equivalences. For $d=2$, the claim follows from  \cite[Thm 8.5]{Horel}. 

In the remaining cases $d\ge3$, we use that $E_d(k)$ is simply connected for all $k$, so by Hurewicz's theorem it suffices to show that $\varphi$ induces an isomorphism on the operad $H_*(\ed)$ in graded abelian groups obtained by taking arity-wise integral homology. We will use two facts about the operad $H_*(\ed)$: firstly, it is degreewise a free abelian group of finite rank (this follows from \cite[III. Lemma 6.2]{Cohen}), so it suffices to show that $\varphi$ induces a surjection in homology. Secondly,  $H_*(\ed)$ is generated under operad compositions in arity $2$ (this follows from the fact that $H_*(\ed)$ is the $d$-Poisson operad, see e.g.\,\cite[Theorem 6.3]{Sinha}). Hence, since $\ed(2)\simeq S^{d-1}$, the operad $H_*(\ed)$ is supported in degrees $H_{t(d-1)}(\ed)$ for $t\ge0$ and $\varphi$ acts in this degree by multiplication with $D^t$ where $D$ is the degree of the induced self-map of $\ed(2)$. The task thus becomes to show $D=\pm1$ which we do by proving that $D$ is not divisible by any prime $p$. If $D$ were divisible by $p$, then by the above discussion $\varphi$ would act by multiplication with $0$ on the reduced $\bfF_p$-homology of $\ed(p)$. In the homological $\bfF_p$-Serre spectral sequence of the fibre sequence $\ed(p)\rightarrow \ed(p)/\Sigma_p\rightarrow B\Sigma_p$, this means that $\varphi$ acts by $0$ on all rows except the bottom one, on which it acts as the identity. This implies that there are no nontrivial differentials out of the bottom row, so the map $\ed(p)/\Sigma_p\rightarrow B\Sigma_p$ is surjective on $\bfF_p$-homology. But this cannot happen since $B\Sigma_p$ has nontrivial $\bfF_p$-homology in arbitrarily high degree and $\ed(p)/\Sigma_p$ is equivalent to a finite-dimensional manifold, namely the configuration space of $p$ \emph{unordered} points in $\bfR^d$. (It may be worth observing that this proof goes through with a cyclic subgroup $C_p \subset \Sigma_p$ in place of $\Sigma_p$, so it does not require the full $\Sigma_p$-equivariance of $\varphi$.)
\end{proof}

\begin{remark}\cref{athm:end-is-aut} fails for several variants of the $\ed$-operad:
\begin{enumerate}
\item It fails in general for the version $E_d^{\theta}$ with  tangential structures: take $\theta$ to be the map $X \to \ast \to \BTop(d)$ and use that any self-map $\psi\colon X\ra X$ induces a self-map of $\ed^\theta$. This is an equivalence if and only if $\psi$ is an equivalence.
\item It fails in general for the localised versions of $E_d$:  there is an endomorphism of the cooperad $H^*(E_d;\bfQ)$ in commutative graded algebras that sends the generator of $H^{d-1}(E_d(2);\bfQ)\cong\bfQ$ to zero. By a version of formality of the rationalised $\ed$-operad $(\ed)_{\bfQ}$ (see \cite[Theorem A, B]{FresseWillwacherFormality} or \cite[Section 12]{BoavidaHorel}), this endomorphism lifts to an endomorphism of $(\ed)_{\mathbf{Q}}$ which is not an equivalence.
\end{enumerate}
\end{remark}

\bibliographystyle{amsalpha}
\bibliography{./refs}

\end{document}